\documentclass[12pt]{amsart}

\usepackage{amsmath, graphicx, cite, enumerate, comment}
\usepackage{amssymb, amsthm, amscd}
\usepackage{latexsym}
\usepackage{amssymb}
\usepackage[english]{babel}
\usepackage{amsmath,amssymb}
\usepackage{amsfonts}
\usepackage{cite}
\usepackage{marvosym}
\usepackage{mathrsfs}
\usepackage{amsthm}
\usepackage{hyperref}
\usepackage{verbatim}

\usepackage{tikz} \usetikzlibrary{arrows,shapes,patterns} 
\usepackage{lipsum} 

\newtheorem{thm}{Theorem}
\newtheorem{lem}[thm]{Lemma}

\newtheorem{cor}[thm]{Corollary}
\newtheorem{prop}[thm]{Proposition}

\theoremstyle{remark}
\newtheorem{rmk}[thm]{Remark}

\theoremstyle{definition}

\numberwithin{equation}{section}

\newcommand{\R}{\mathbb{R}}

\newcommand{\be}{\begin{equation}}
\newcommand{\ee}{\end{equation}}

\newcommand{\ben}{\begin{equation*}}
\newcommand{\een}{\end{equation*}}

\def\barroman#1{\sbox0{#1}\dimen0=\dimexpr\wd0+1pt\relax
  \makebox[\dimen0]{\rlap{\vrule width\dimen0 height 0.06ex depth 0.06ex}%
    \rlap{\vrule width\dimen0 height\dimexpr\ht0+0.03ex\relax 
            depth\dimexpr-\ht0+0.09ex\relax}%
    \kern.5pt#1\kern.5pt}}

\def\XXint#1#2#3{{\setbox0=\hbox{$#1{#2#3}{\int}$}
     \vcenter{\hbox{$#2#3$}}\kern-.5\wd0}}

\newcommand{\h}{\mathscr{H}}
\newcommand{\A}{\mathcal{A}}

\newcommand{\emb}{\hookrightarrow}

\newcommand{\sm}{{\setminus}}

\def\XXint#1#2#3{{\setbox0=\hbox{$#1{#2#3}{\int}$}
     \vcenter{\hbox{$#2#3$}}\kern-.5\wd0}}

\begin{document} 	               
       \thanks{\textsl{Mathematics Subject Classification (MSC 2010):} Primary 53A10; Secondary 53C42, 49Q05.}
     \title[Geometric convergence results for minimal surfaces via bubbling analysis]{Geometric convergence results for closed minimal surfaces via bubbling analysis}
     \author{Lucas Ambrozio, Reto Buzano, Alessandro Carlotto and Ben Sharp}
     \address{ \noindent L. Ambrozio: IMPA - Associa\c{c}\~{a}o Instituto Nacional de Matem\'atica Pura e Aplicada, Rio de Janeiro, RJ, Brazil, 22460-320, \textit{E-mail address: l.ambrozio@impa.br}
     	\newline \newline 
     	\indent R. Buzano: School of Mathematical Sciences, Queen Mary University of London, London E1 4NS, United Kingdom, \textit{E-mail address: r.buzano@qmul.ac.uk} 
	\newline \indent Universit\`a degli Studi di Torino, Dipartimento di Matematica, 10123 Torino, Italy, \textit{E-mail address: reto.buzano@unito.it}
     	\newline\newline
     	\indent A. Carlotto: ETH - Department of Mathematics, 8092 Z\"urich, Switzerland, \textit{E-mail address: alessandro.carlotto@math.ethz.ch}
     	 \newline \newline 
     \indent B. Sharp: School of Mathematics,
     University of Leeds, Leeds LS2 9JT, United Kingdom \textit{E-mail address: B.G.Sharp@leeds.ac.uk}}

     	\begin{abstract}We present some geometric applications, of global character, of the bubbling analysis developed by Buzano and Sharp for closed minimal surfaces, obtaining smooth multiplicity one convergence results under upper bounds on the Morse index and suitable lower bounds on either the genus or the area. For instance, we show that given any Riemannian metric of positive scalar curvature on the three-dimensional sphere the class of embedded minimal surfaces of index one and genus $\gamma$ is sequentially compact for any $\gamma\geq 1$.
     		
     			Furthemore, we give a quantitative description of how the genus drops as a sequence of minimal surfaces converges smoothly, with mutiplicity $m\geq 1$, away from finitely many points where curvature concentration may happen. This result exploits a sharp estimate on the multiplicity of convergence in terms of the number of ends of the bubbles that appear in the process.
     			\end{abstract}
     	
     	\maketitle    
     	
     	\section{Introduction}
     	
     	Let $(N^3,g)$ be a compact Riemannian manifold of dimension three, without boundary. We shall be concerned here with certain global phenomena related to the convergence of a sequence of closed minimal surfaces, smoothly embedded in $N$, of bounded area and index. To that end, let us introduce the following notations:
     	\[
     		\mathfrak{M}(\Lambda,I) := \{ M \in \mathfrak{M} \ : \ \h^2(M) \leq \Lambda,\ index(M) \leq I\}
     			\]
     			is the set of closed, connected, smooth and embedded minimal surfaces (denoted by $\mathfrak{M}$) with area and Morse index bounded from above, and for $p$ an integer greater or equal than one
     \[
     			\mathfrak{M}_p(\Lambda,\mu) := \{ M \in \mathfrak{M} \ : \ \h^2(M) \leq \Lambda,\ \lambda_p \geq -\mu\}
     			\]
     			is the set of closed, connected, smooth and embedded minimal surfaces with bounded area and $p^{th}$ eigenvalue $\lambda_p$ of the Jacobi operator  bounded from below.

     	In \cite{Sha15}, the fourth-named author proved a compactness theorem for the set $\mathfrak{M}(\Lambda,I)$, and in later joint work Ambrozio-Carlotto-Sharp \cite{ACS15} also proved a similar compactness theorem for $\mathfrak{M}_p(\Lambda,\mu)$: given a sequence of minimal surfaces $\left\{M_k\right\}$ in $\mathfrak{M}(\Lambda,I)$ (or $\mathfrak{M}_p(\Lambda,\mu)$), there is some smooth limit in the same class to which the sequence sub-converges smoothly and graphically (with finite multiplicity $m$) away from a discrete set $\mathcal{Y}$ on the limit, where one witnesses the formation of necks. In order to describe the local picture in more detail, let us agree
     	to employ the word \textsl{bubble} to denote a complete, embedded, connected and non-flat minimal surface of finite total curvature in $\R^3$, namely satisfying $\mathcal{A}(\Sigma):=\int_{\Sigma}|A|^2\,d\h^2<\infty$.

     	The aforementioned concentration phenomenon was carefully analyzed by Buzano-Sharp in \cite{BS17} and their main result implies that, in the setting above, associated with each point $y\in\mathcal{Y}$ there is a finite and positive number $J_y$ of \textsl{bubbles} $\Sigma^{y}_{\ell}$ (for $\ell=1,\ldots, J_y$), that are suitable blow-up limits of the sequence at the concentration points, and such that the following identity holds:
  	\begin{equation}\label{eq:quant}
     	\lim_{k\to \infty} \A(M_k) = m\A(M) + \sum_{y\in \mathcal{Y}}\sum_{\ell=1}^{J_y} \A(\Sigma^y_\ell).
     	\end{equation}

     		Since we are working in ambient dimension three, we can combine the previous \textsl{quantization identity} with the Gauss-Bonnet theorem to obtain interesting information relating the topology of $M_k$ (with $k$ large) with that of the limit surface $M$ and of the bubbles that arise in the previous analysis. Specifically, one can derive from \eqref{eq:quant} the equation
     		\begin{equation}\label{eq:geomquant}
     		\chi(M_k)=m\chi (M)+\sum_{y\in\mathcal{Y}}\sum_{\ell=1}^{J_y} (\chi(\Sigma^y_{\ell})-b^y_{\ell})
     		\end{equation}
     		where $b^y_{\ell}$ is the number of ends of $\Sigma^y_{\ell}$.

     			A more detailed summary of these results and a short discussion of our conventions concerning the Euler characteristic of open and non-orientable surfaces, to the extent that is needed in the present paper, is given in Section \ref{sec:ancill}.
     			 
     		We will first employ the identity above together with another ancillary result, Lemma \ref{lem:sumindex}, to prove a strong compactness theorem for minimal surfaces of bounded index inside a 3-manifold of positive \textsl{scalar} curvature. By the work of Chodosh-Ketover-Maximo \cite{CKM15}, such sequences have uniformly bounded area, so that we are actually in the situation described above.

	\begin{thm}\label{thm:conv1}
		Let $(N^3,g)$ be a compact Riemannian manifold, without boundary, of positive scalar curvature. For a fixed integer $j\in\left\{0,1,2,3,4\right\}$ let $\left\{M_k\right\}$ be a sequence of closed, embedded minimal surfaces with Morse index bounded from above by $j$. If
		\begin{enumerate}
			\item[\underline{\textsl{either}}]{$\chi(M_k)<2-2j$ for all $k\in \mathbb{N}$,}
			\item[\underline{\textsl{or}}]{$\chi(M_k)< 4-2j$ for all $k\in \mathbb{N}$, and $N^3$ does not contain any  embedded minimal $\R\mathbb{P}^2$ (which happens, for instance, if $N^3$ is simply connected),}
		\end{enumerate}
		then, up to extracting a subsequence, one has that $\left\{M_k\right\}$ converges smoothly to some closed, embedded minimal surface $M$, of Morse index bounded from above by $j$, with multiplicity one.
	\end{thm}	
	
	\begin{rmk} To our knowledge, a result of this type is new even in the stable case, i.~e.\! when $j=0$. Indeed, what would be standard to prove, in the setting above, is that there must be smooth subsequential convergence at all points (that is to say $\mathcal{Y}=\emptyset$), but possibly with multiplicity $m=2$. In fact, well-known examples show that the statement above is sharp, in the sense that one can have sequences of stable minimal spheres converging with multiplicity two to a minimal projective plane. For instance, consider any three manifold containing an open subdomain that is isometric to the Riemannian product $S^2\times (-1,1)/\simeq$ where 
		\[
		(x,t)\simeq (x',t') \ \Leftrightarrow \ x'=-x, t'=-t. \ 
		\]
		Then any sphere $M_{t_0}$ described, in these coordinates, by $t=t_0$ for $t_0\in (0,1)$ is stable, totally geodesic and for any sequence $t_k\searrow 0$ one has that $M_{t_k}$ converges to $M_0\equiv\R\mathbb{P}^2$ with multiplicity two.
	\end{rmk}
	
	By the work of Choi-Schoen \cite{CSc85} we know that if $(N^3,g)$ has positive \textsl{Ricci} curvature, then a topological bound suffices to gain strong convergence with multiplicity one. However, one cannot expect the same conclusion to hold in the much broader setting of 3-manifolds of positive scalar curvature, as the above example shows. In fact, Colding and De Lellis presented in \cite{CdL05} a method to construct 3-manifolds of positive scalar curvature containing sequences of embedded, orientable minimal surfaces of any fixed genus $\gamma$, that converge to a minimal lamination with two singular points on a strictly stable minimal sphere.  Still, Theorem \ref{thm:conv1} shows how an extra assumption (e.~g.\! on the index) suffices to gain strong compactness as in the result by Choi and Schoen. A simple application of such a theorem is presented in the following corollary:
	\begin{cor}\label{cor:3sphere}
		Let $g$ be a Riemannian metric of positive scalar curvature on the three-dimensional sphere. The class of stable, embedded minimal surfaces is sequentially compact in the sense of smooth multiplicity one convergence. Similarly, the class of embedded minimal surfaces of index one and genus $\gamma$ is sequentially compact for any $\gamma\geq 1$ in the sense of smooth multiplicity one convergence.
	\end{cor}	
	
	We can also prove a counterpart of Theorem \ref{thm:conv1} which applies to minimal surfaces with a uniform, but possibly large, index bound. Yet, the topological lower bound one has to assume needs to be stronger due to the lack of classification results for complete (embedded) minimal surfaces $\Sigma\subset\R^3$ of index equal to any natural number greater or equal than four. We will circumvent this obstacle by exploiting the index estimates obtained by Chodosh and Maximo in \cite{CM14, CM18b}. We refer the reader to Section \ref{sec:comp} for the corresponding statement, see Theorem \ref{thm:conv2}.
			
			\
			
			The two results above can both be regarded as instances of a local-to-global correspondence, meaning that the understanding of complete minimal surfaces in the Euclidean space $\R^3$ is exploited to extract information for the blow-up analysis at the singular points of the convergence process, which in turn is needed to derive novel information on the space of minimal cycles inside a given Riemannian manifold. Compactness theorems of similar spirit have appeared (in the Euclidean setting) in the pioneering work by Ros \cite{R95} about the Hoffman-Meeks conjecture \cite{HM90b} and, relying on somehow different methods, by Traizet \cite{Tra04}.
			
			 One motivation for us to transplant those ideas to curved ambient spaces was provided by the recent, remarkable advances in the construction of closed, embedded minimal hypersurfaces in general ambient manifolds, either via min-max methods in the spirit of Almgren-Pitts \cite{P81} as developed by Marques and Neves, or studying the interfaces arising as suitable singular limits of solutions to the Allen-Cahn equation as proposed by Guaraco \cite{Gua18}. In particular, both approaches have been successful to prove the existence of infinitely many minimal surfaces in 3-manifolds under the assumption that the ambient Ricci curvature be positive (see \cite{MN13} by Marques and Neves, and \cite{GG16} by Gaspar and Guaraco) or instead under the assumption that the ambient metric be generic (see the works by Marques and Neves with K. Irie \cite{IMN18} and A. Song \cite{MNS17}, relying on earlier work on Weyl's law for the volume spectrum with Liokumovich  \cite{LMN18}, for the former technique and the very recent work \cite{CM18} by Chodosh and Mantoulidis for the latter one). Building on these advances, A. Song \cite{Son18} was then able to obtain an unconditional existence result, thereby providing a proof of Yau's conjecture (cf. \cite{Yau82}).

			As a different but related application, we further employ the bubbling analysis to prove a topological semicontinuity result, which formalizes the well-known intuition that the genus \textsl{can only drop} as a sequence of minimal surfaces converges smoothly away from a finite concentration set, as described above.

			\begin{thm}\label{thm:lsc}
				Let $(N^{3},g)$ be a smooth, compact, orientable Riemannian manifold without boundary.
				Consider a sequence of closed, orientable, embedded minimal surfaces $\{M_k\}\subset \mathfrak{M}_p(\Lambda,\mu)$ for some fixed constants $\Lambda\in \R$, $\mu\in \R$ independent of $k$, and assume it has an orientable limit
				 $M\in \mathfrak{M}_p(\Lambda,\mu)$, in the sense of smooth graphical convergence with multiplicity $m\geq 1$ away from a finite set $\mathcal{Y}$ of points. 
				Then for all sufficiently large $k\in\mathbb{N}$ one has
				\begin{equation}\label{eq:lsc}
				m\cdot  genus(M) + \sum_{y\in\mathcal{Y}}\sum_{\ell=1}^{J_y} genus (\Sigma^y_\ell) \leq genus (M_k).
				\end{equation}
				The inequality above is strict unless
				\[
				m=1+\sum_{y\in\mathcal{Y}}\sum_{\ell=1}^{J_y}(b^y_{\ell}-1).
				\]
			\end{thm}
			
				\begin{rmk}\label{rem:equal}
					With respect to the previous statement, notice that if equality occurs then there are always \textsl{at most} $m-1$ bubbles, and if there are \textsl{exactly} $m-1$ bubbles then they must all have two ends and thus be catenoids by \cite{Sc83}. In particular, if $m=2$ and equality holds, then $\mathcal{Y}=\left\{y\right\}, J_y=1$ and the only bubble is a catenoid.
					
				\end{rmk}
			
			When $N^3$ is not assumed to be orientable, or when the surfaces $M_k$ or the limit $M$ are allowed to be non-orientable, one can still recover results in the same spirit, see Remark \ref{rem:nonorient}.
			
				As a direct consequence of the statement above, we can rigorously justify the fact that genus strictly drops when there is at least one point of bad convergence:
				
				\begin{cor}\label{cor:folk}
					In the setting of Theorem \ref{thm:lsc}, if $\liminf_{k\to\infty}genus(M_k)\geq 1$ and the convergence of $M_k$ to $M$ is \textsl{not} smooth (i.~e.\! if $|\mathcal{Y}|>0$) then $genus(M)<\liminf_{k\to\infty}genus(M_k)$. In particular, if  	$\liminf_{k\to\infty}genus(M_k)=1$ then $genus(M)=0$, i.~e.\! $M\simeq S^2$.
					\end{cor}	
					
				Furthermore, we can complement the content of Corollary \ref{cor:3sphere} by giving a classification of the possible limit pictures describing sequences of minimal spheres of index not greater than one.	
				
				\begin{cor}\label{cor:3sphereBIS}
					Let $g$ be a Riemannian metric of positive scalar curvature on the three-dimensional sphere and let $\left\{M_k\right\}$ be a sequence of embedded minimal spheres of index one. Then: \textsl{either} a subsequence converges smoothly, with multiplicity one, to an embedded minimal sphere of index at most one \textsl{or} there exists a subsequence converging smoothly, with multiplicity two and exactly one catenoidal bubble, to a stable embedded minimal sphere. In particular, if $(S^3,g)$ does not contain any stable minimal sphere then the class of index one embedded minimal spheres is sequentially compact in the sense of smooth multiplicity one convergence. 
				\end{cor}	
				
				Indeed, if a subsequence converged with multiplicity $m\geq 2$ then by Theorem \ref{thm:lsc} the limit minimal surface should also be a sphere, all bubbles should have genus zero and equality holds. This information, combined with Lemma \ref{lem:sumindex} implies that $m=2$ and thus the conclusion follows from Remark \ref{rem:equal} above. In fact, in the setting of Corollary \ref{cor:3sphereBIS} the sole assumption of a uniform index bound allows to prove, invoking the classification theorem by Lopez and Ros \cite{LR91}, that when bubbling occurs there is subsequential convergence with multiplicity $m$ and exactly $m-1$ catenoidal bubbles (as in Remark \ref{rem:equal}).
					
				The proof of Theorem \ref{thm:lsc} crucially relies on a sharp multiplicity estimate, Proposition \ref{pro:mult}, which is presented in Section \ref{sec:ancill}. Roughly speaking, we gain (in the setting of Theorem \ref{thm:main}) an effective control on the integer $m$ in terms of the topological data of the bubbles, in fact only involving the number of their ends. Hence, one can also use this proposition to derive asymptotic area estimates for the sequence $\left\{M_k\right\}$ in terms of their index and of a lower bound on the scalar curvature of the ambient manifold $(N^3,g)$. This can be turned into a compactness theorem in the same spirit of Theorem \ref{thm:conv1} and Corollary \ref{cor:3sphere}.

			\begin{thm}\label{thm:conv1'}
				Let $(N^3,g)$ be a compact Riemannian manifold, without boundary, of scalar curvature bounded below by some constant $\rho>0$. For a fixed integer $j\in\mathbb{N}$ let $\left\{M_k\right\}$ be a sequence of closed, embedded minimal surfaces with Morse index bounded from above by $j$. If
				\[
				\limsup_{k\to\infty}\h^2(M_k)> \frac{8 \pi(1+j)}{\rho}
				\]
				then, up to extracting a subsequence, one has that $\left\{M_k\right\}$ converges smoothly to some closed, embedded minimal surface $M$, of Morse index bounded from above by $j$, with multiplicity one.
			\end{thm}
			
			This statement connects in particular with the study of minimal surfaces of index one obtained via one-dimensional min-max schemes, hence with the notion of width of a Riemannian manifold. If we assume, as a convenient normalization, that $(N^3,g)$ has scalar curvature bounded from below by $6$, then the threshold that is prescribed by the previous theorem for $j=1$ is $8\pi/3$, to be compared with the results in \cite{MN12} asserting that when $N^3$ is diffeomorphic to $S^3$ and satisfies additional geometric assumptions (see for instance Theorem 1.1 and Theorem 4.9 therein) the width is always bounded from above by $4\pi$, namely the value of the area of any equatorial two-sphere in the round three-dimensional sphere.

			\
			
			Analogous theorems can also be obtained in the case of free boundary minimal surfaces, which are the object of our article \cite{ABCS18}, following the general compactness analysis presented by Ambrozio, Carlotto and Sharp in \cite{ACS17}.
			
			\
			
			\textsl{Acknowledgements.} The authors would like to thank Andr\'e Neves for suggesting the question which inspired this research project. During the preparation of this article, L. A. was supported by the EPSRC on a Programme Grant entitled `Singularities of Geometric Partial Differential Equations' reference number EP/K00865X/1.
		 This project was completed while A. C. was a visiting professor at the Scuola Normale Superiore, and he would like to thank the faculty and staff members for the warm hospitality and excellent working conditions.
		 Some results have been sharpened in the revision process based on the improved index estimate in \cite{CM18b}, that appeared after a first version of this article had been submitted. We thank the anonymous referee for suggesting these revisions.
			
			\section{Some ancillary results}\label{sec:ancill}
			
			Let us start by reviewing the bubbling analysis presented in \cite{BS17}, whose main result (specified to the case of ambient dimension three) can be summarized as follows:

			\begin{thm}\label{thm:main}
				Let $(N^{3},g)$ be a smooth, compact Riemannian manifold without boundary. If $\{M_k\}\subset \mathfrak{M}_p(\Lambda,\mu)$ for some fixed constants $p\geq 1,\Lambda\in \R$, $\mu\in \R_{\geq 0}$ independent of $k$, then up to subsequence there exist $M\in \mathfrak{M}_p(\Lambda,\mu)$ and $m\in \mathbb{N}$ where $M_k \to m M$ in the varifold sense and a set $\mathcal{Y} = \{y_i\}\subset M$ of at most $p-1$ points such that the convergence to $M$ is smooth and graphical (with multiplicity $m$) away from $\mathcal{Y}$.

				Moreover, associated with each $y \in \mathcal{Y}$ there exists a finite number $0<J_y\in \mathbb{N}$ of bubbles $\{\Sigma_\ell^y\}_{\ell=1}^{J_y}$ with $\sum_y J_y \leq p-1$ as well as associated point-scale sequences $\{(p^{y,\ell}_k, r^{y,\ell}_k)\}_{\ell=1}^{J_y}$ with $p^{y,\ell}_k \to y$ for all $\ell$, and $r^{y,\ell}_k\to 0$, so that:
				
				\begin{enumerate}
				\item	For all $y\in\mathcal{Y}$ and $1\leq i\neq j\leq J_y$ 
						\[\frac{dist_g(p^{y,i}_k,p^{y,j}_k)}{r^{y,i}_k + r^{y,j}_k}\to \infty.
						\]
						Taking normal coordinates centred at $p^{y,\ell}_k$ and letting
						$\tilde{M}^{y,\ell}_k:=M_k/r^{y,\ell}_k \subset \R^{3}$ then $\tilde{M}^{y,\ell}_k$ converges smoothly on compact subsets to $\Sigma^y_\ell$ with multiplicity one.

						\noindent Moreover, given any other sequence $M_k \ni q_k$ and $\varrho_k \to 0$ with $q_k \to \overline{y}\in\mathcal{Y}$ and 
						\[\min_{y\in\mathcal{Y}}\min_{\ell=1,\dots,J_y} \Big(\frac{\varrho_k}{r^{y,\ell}_k} + \frac{r^{y,\ell}_k}{\varrho_k} + \frac{dist_g(q_k,p^{y,\ell}_k)}{\varrho_k + r^{y,\ell}_k}\Big)\to \infty,
						\]
						taking normal coordinates at $q_k$ and letting 
						$\hat{M}_k:= M_k/\varrho_k\subset \R^{3}$
						then $\hat{M}_k$ converges smoothly on compact subsets to a collection of parallel planes. 
					\item The following equation holds
					\[
					\lim_{k\to \infty} \A(M_k) = m\A(M) + \sum_{y\in \mathcal{Y}}\sum_{\ell=1}^{J_y} \A(\Sigma^y_\ell)
					\]
					where we have denoted by $\mathcal{A}(M)$ and $\mathcal{A}(M_k)$ the total curvature in $(N,g)$ of the minimal surfaces $M$ and $M_k$, respectively.
				\end{enumerate}
				Furthermore when $k$ is sufficiently large, the surfaces $M_k$ of this subsequence are all diffeomorphic to one another.
			\end{thm}

			\begin{rmk}
				Since patently $\mathfrak{M}(\Lambda,I)= \mathfrak{M}_{I+1}(\Lambda,0)$ the previous assertion also implies an analogous result for the space of minimal surfaces in $(N^3,g)$ with a uniform upper bound on the area and on the Morse index.	
			\end{rmk}	
			
			As stated in the introduction, one can then easily derive from the previous result equation \eqref{eq:geomquant}, which will turn out to be very useful for the scopes of the present paper. To avoid ambiguities let us briefly recall some elementary facts and our conventions. As it is well known, the topology of a compact orientable surface is completely described by its genus; for compact non-orientable surfaces, we define their genus to be equal to the genus of their orientable double cover. According to that convention, the Euler characteristic of a compact surface $M$ of genus $\gamma$ is given by $\chi(M)=2-2\gamma$ if $M$ is orientable, and $\chi(M)=1-\gamma$ otherwise.
			Also, it is well-known (see \cite{Oss63, Oss64}) that bubbles are orientable and have finite topology, namely they are homeomorphic to a compact orientable surface of genus $\gamma$ minus a finite number $b$ of points (corresponding to the ends of $\Sigma$), and their Euler characteristic is given by $\chi(\Sigma)=2-2\gamma-b$. 
			
			\begin{cor}\label{lem:geomquant}
				(Setting as in Theorem \ref{thm:main}). There exists an infinite subset $K\subset\mathbb{N}$ depending on the ambient manifold $(N^3,g)$ and on the sequence $\left\{M_k\right\}$ such that for any $k\in K$ one has
				\[
				\chi(M_k)=m\chi (M)+\sum_{y\in\mathcal{Y}}\sum_{\ell=1}^{J_y} (\chi(\Sigma^y_{\ell})-b^y_{\ell})
				\]
				where $b^y_{\ell}$ is the number of ends of $\Sigma^y_{\ell}$. 
			\end{cor}

			\begin{proof}
				The Gauss equation for the minimal surface $M$ in the ambient manifold $(N^3,g)$ takes the form $|A|^2=2 \textrm{Sec}_{TM}-2K$ where $K$ stands for the Gauss curvature of the surface in question and $\textrm{Sec}_{TM}$ denotes the sectional curvature of $(N^3,g)$ along the plane spanned by a pair of linearly independent tangent vectors to $M$. Thus, integrating this equation, we get by virtue of the Gauss-Bonnet theorem
				\[
				\mathcal{A}(M)=- 4\pi\chi(M)+2\int_M \textrm{Sec}(TM)\,d\h^2;
				\]
				similarly for each minimal surface $M_k$
				\[
				\mathcal{A}(M_k)=-4\pi\chi(M_k)+2\int_{M_k} \textrm{Sec}(TM_k)\,d\h^2;
				\]
				and also 
				$
				\mathcal{A}(\Sigma)=-4\pi(\chi(\Sigma^y_\ell)-b^y_{\ell})
				$
				for each bubble $\Sigma^y_{\ell}\subset\R^3$  (cf. \cite{JM83}).
				Thus, since clearly the varifold convergence $M_k\to m M$ implies
				\[
				\int_{M_k} \textrm{Sec}(TM_k)\,d\h^2\to m \int_M \textrm{Sec}(TM)\,d\h^2,
				\]
				the equation \eqref{eq:quant} can be rewritten in the form
				\[
				\lim_{k\to\infty} \chi(M_k)=m\chi(M)+\sum_{y\in \mathcal{Y}}\sum_{\ell=1}^{J_y}(\chi(\Sigma^y_\ell)-b^y_{\ell})
				\]
				which implies the claim.
			\end{proof}
			
			For the proof of Lemma \ref{lem:sumindex}, we need two preliminary observations.
			
			\
			
				We first recall from \cite{FC85} and \cite{Sc83} that an embedded minimal surface of finite total curvature $\Sigma\subset\R^3$ has finite Morse index (in fact these two finiteness conditions are equivalent) and is \textsl{regular at infinity} meaning that it can be decomposed, outside a compact set of $\R^3$ as a finite union of graphs with a suitable asymptotic expansion. In particular, one can certainly find $R>0$ large enough that
				\[
			index(\Sigma \cap B_R(0))=index(\Sigma) \,\,\,\,\,\text{and} \,\,\,\,\,\, 	genus(\Sigma \cap B_R(0))=genus(\Sigma).
				\]
			
			Furthermore, it follows from part $(1)$ of the statement of Theorem \ref{thm:main} that for all $R>0$, there exists $k_0$ such that for all $k>k_0$, the set 
				\[
				\{B^{(N,g)}_{R r^{y,\ell}_k}(p^{y,\ell}_k)\}^{J_y}_{\ell=1}
				\] consists of pairwise disjoint balls. Indeed, if that were not the case there would exist $R_0$, indices $1\leq i\neq j\leq J_y$ and an infinite subset $K\subset\mathbb{N}$ so that for all $k\in K$
				\[
				B^{(N,g)}_{R_0 r^{y,i}_k}(p^{y,i}_k)\cap B^{(N,g)}_{R_0 r^{y,j}_k}(p^{y,j}_k)\neq \emptyset
				\]
			hence
				\[\frac{dist_g(p^{y,i}_k,p^{y,j}_k)}{r^{y,i}_k + r^{y,j}_k} \leq 2R_0
				\]
			which is a contradiction.

			\begin{lem}\label{lem:sumindex}

				(Setting as in Theorem \ref{thm:main}). 
				\begin{equation}\label{eq;sumindex}
				\sum_{y\in\mathcal{Y}}\sum_{\ell=1}^{J_y} index(\Sigma^y_{\ell})\leq \liminf_{k\to\infty} index(M_k).
				\end{equation}
			\end{lem}	
			
			\begin{proof} 
				Let us consider, without loss of generality, an infinite subset $K\subset\mathbb{N}$ such that if $\underline{k}\in K$ then
			$index(M_{\underline{k}})=\liminf_{k\to\infty}index(M_k)$.
				Let $R>0$ be chosen once and for all, based on the remarks we have presented before the statement of this lemma, so that 
				\[
				index(\Sigma^y_{\ell} \cap B_R(0))=index(\Sigma^y_{\ell})
				\]
				for all $y\in\mathcal{Y}, 1\leq \ell\leq J_y$, namely for each one of the finitely many bubbles mentioned in the statement of Theorem \ref{thm:main}, applied to the subsequence of minimal surfaces $\left\{M_k\right\}_{k\in K}$. Then for such a choice of $R$ we have that for $k$ belonging to a further subsequence (so $k\in K'\subset K\subset \mathbb{N}$)
				\[
				index (M_k \cap B^{(N,g)}_{R r^{y,\ell}_k}(p^{y,\ell}_k))= index (\tilde{M}^{y,\ell}_k \cap B^{(\tilde{N},\tilde{g})}_R(0)) \geq index(\Sigma^y_\ell\cap B_R(0)) = index(\Sigma^y_\ell)
				\]
				where the inequality relies on the smooth convergence, multiplicity one, of $\tilde{M}^{y,{\ell}}_k$ to $\Sigma^y_{\ell}$.
				Here we have denoted by $(\tilde{N},\tilde{g})$ the (locally defined) smooth manifold which is obtained by scaling $(N,g)$ by a factor $r^{y,\ell}_k$, where the operation is understood in normal coordinates centered at the point $p^{y,\ell}_k$ (the explicit dependence on the scaling parameter is omitted for notational convenience).
				
				On the other hand, by virtue of the remarks we presented before the statement of this lemma, for all $k\in K'$ large enough
				\[
				index(M_k)\geq \sum_{y\in\mathcal{Y}} \sum_{\ell=1}^{J_y}  index (M_k \cap B^{(N,g)}_{R r^{y,\ell}_k}(p^{y,\ell}_k)),
				\]
				which completes the proof.
			\end{proof}

			The same argument presented above for Lemma \ref{lem:sumindex}, this time relying on the elementary inequality 
			\[
		genus(M_k)\geq \sum_{y\in\mathcal{Y}} \sum_{\ell=1}^{J_y}  genus (M_k \cap B^{(N,g)}_{R r^{y,\ell}_k}(p^{y,\ell}_k)),
		\]
				which holds true since the open balls in question are disjoint, shows that 
			\begin{equation}\label{eq;sumgenus}
			\sum_{y\in\mathcal{Y}}\sum_{\ell=1}^{J_y} genus(\Sigma^y_{\ell})\leq \liminf_{k\to\infty} genus(M_k).
			\end{equation}

			Our goal however is to prove a much stronger statement, Theorem \ref{thm:lsc}, which we can do at the cost of a more delicate proof. 
			To that end, the key step is the following multiplicity estimate (of independent interest and applicability). 
			To state it in a concise fashion, we need to remind the reader of the construction of a twofold cover of a one-sided minimal surface (see e.~g.\! Section 6 of \cite{ACS17}). 
			
			If we let $f:\tilde{M}\to N$
			be the two-sided minimal immersion associated to $M$, we consider the associated 
			pulled-back bundle, $f^*NM$, which is trivial (by definition of $\tilde{M}$) and whose zero section describes $\tilde{M}$. A sufficiently small neighbourhood of the zero section, denoted $\tilde{U}$, is in a two-to-one correspondence with a small tubular neighbourhood $U$ of the one-sided surface $M$ in $N^3$. Therefore, we may pull back the metric on $U$ and see $\tilde{M}\emb\tilde{U}$ as a two-sided embedded minimal surface. For large enough $k$ (so that eventually $M_k$ lies inside this tubular neighbourhood of $M$) we can equally consider the pull back of $M_k$, denoted $\tilde{M}_k\emb \tilde{U}$ which is again an embedded minimal surface (possibly disconnected, but with at most two components). Nevertheless, we still have $\tilde{M}_k \to m\tilde{M}$ locally smoothly and graphically on $\tilde{M}\sm\tilde{\mathcal{Y}}$ with $|\tilde{\mathcal{Y}}|=2|\mathcal{Y}|$, and we have two copies of each of the original bubbles appearing in the convergence in $\tilde{U}$.

			\begin{prop}\label{pro:mult}
				(Setting as in Theorem \ref{thm:main}).
				\begin{enumerate}
					\item If $M$ is two-sided then 
					$$m\leq 1 + \sum_{y\in\mathcal{Y}}\sum_{\ell=1}^{J_y} (b^y_\ell - 1).$$
					\item If $M$ is one-sided then  
					\begin{itemize}
						\item when $\tilde{M}_k$ is connected then
						$$m\leq 1 + 2\sum_{y\in\mathcal{Y}}\sum_{\ell=1}^{J_y} (b^y_\ell - 1).$$
						\item when $\tilde{M}_k$ is not connected then 
						$$m\leq 2 + 2\sum_{y\in\mathcal{Y}}\sum_{\ell=1}^{J_y} (b^y_\ell - 1).$$
					\end{itemize}
					\end{enumerate}
				\end{prop}
			
			\begin{proof}
				We will only prove part $(1)$, as case $(2)$ can be established by a simple variation of the same argument.
				
				\
				
				The conclusion is trivial if the convergence happens with multiplicity $m=1$ (for, in this case, we already know from \cite{Sha15} that $\mathcal{Y}=\emptyset$ and thus there are no bubbles at all), so let us assume instead $m\geq 2$, which implies that $|\mathcal{Y}|>0$ because $M$ is assumed to be two-sided.
				
				 By the neck analysis presented in Section 4 of \cite{BS17} and in particular claim 1 therein, we see that, for any $R$ fixed and sufficiently large, there exists $k_0\in\mathbb{N}$ such that whenever $k\in K, k\geq k_0$ each connected component of
				 \[
				 M_k \sm \bigcup_{y\in\mathcal{Y}}\bigcup_{\ell =1}^{J_y} B^{(N,g)}_{R r^{y,\ell}_k}(p^{y,\ell}_k)
				 \]
				 is a normal graph over some region in $M$ (with suitable estimates). Here we have conveniently denoted by $K\subset\mathbb{N}$ an infinite subset such that the convergence of $\left\{M_k\right\}$ to $M$ as we let $k\to\infty, k\in K$ satisfies all of the conclusions of Theorem \ref{thm:main}.

				\
				
				 Furthermore, by picking normal coordinates centred at $p^{y,\ell}_k$, one has that the intersection 
				 \begin{equation}\label{eq:intersect}
				 M_k\cap  B^{(N,g)}_{R r^{y,\ell}_k}(p^{y,\ell}_k)
				 \end{equation} is as close as we like to the blown-down bubble $R r^{y,\ell}_k \Sigma^y_\ell$. For the sake of brevity, we shall refer to
				 \[
				 B^{(N,g)}_{R r^{y,\ell}_k}(p^{y,\ell}_k)
				 \] as \textsl{bubble region} and to the intersection \eqref{eq:intersect} as an \textsl{almost bubble}, and notice that we can refer to its ``ends'' in the same way that we would do for $\Sigma^y_\ell$.

				Since $M$ is two-sided, we can order the sheets (the images of the aforementioned graphs) from the lowest to highest (after choosing a unit normal), and the idea of the proof is to construct a path in $M_k$ passing from one sheet to the next via the almost bubbles. 
				From here on, we can assume $k\in K, k\geq k_0$ and fixed (the same argument continues to work for all larger $k$).

				Start at a point in the lowest sheet and move (staying within this graph) until we enter a bubble region which we re-label $B^{(N,g)}_{r_1}(p^1)$ (this must happen, otherwise $M_k$ is disconnected). 
				Once inside the bubble region, move up the almost bubble till we reach the highest  end, then move outside of the bubble region. Notice that at this point we have moved up to the $(b_1)^{th}$ sheet.   
					Now we face a dichotomy:
				\begin{enumerate}
					\item[\textsl{(i)}] \textsl{\underline{either}} we have reached the top sheet,
					\item [\textsl{(ii)}] \textsl{\underline{or}} we have not reached the top sheet. 
				\end{enumerate}
				In case $(i)$, then the process stops here. If $(ii)$ we move on the sheet we are on, until we reach another bubble region $B^{(N,g)}_{r_2}(p^2)$ for which both of the following assertions hold: 
				\begin{itemize}
					\item $B^{(N,g)}_{r_1}(p^1)\cap B^{(N,g)}_{r_2}(p^2)=\emptyset$ and
					\item we are not on the top sheet of the almost bubble inside $B^{(N,g)}_{r_2}(p^2)$. 
				\end{itemize}
				There must be another bubble region satisfying the above, otherwise $M_k$ would be disconnected. 
				Once inside the second bubble region satisfying the above, we again move upwards until we are on the top sheet of the second almost bubble. Notice that this corresponds to moving up \emph{at most} $b_2 -1$ further sheets. Again, move outside of the bubble region and face the dichotomy. 
				
				Continue inductively until we satisfy $(i)$ in the dichotomy. This gives a collection of disjoint bubble regions that we have moved through $\{B^{(N,g)}_{r_i}(p^i)\}_{i=1}^I$, and furthermore 
				\[m\leq b_1 + (b_2 -1) + \dots + (b_I -1).
				\]
				Thereby the proof is complete. 
			\end{proof}
			
			We proceed with a simple result characterizing stable, but possibly one-sided, minimal surfaces \emph{arising as limits} inside 3-manifolds of positive scalar curvature.

			\begin{lem}\label{lem:stable-type}
				Let $(N^3,g)$ be a compact Riemannian manifold, without boundary, of positive scalar curvature. Let $M\subset N^3$ be a closed, embedded minimal surface arising as the limit (in the sense of smooth convergence with multiplicity $m\geq 2$ away from a finite set $\mathcal{Y}$ of points) of a sequence of closed embedded minimal surfaces in $(N^3,g)$. Then $M$ is diffeomorphic to either $S^2$ or $\R\mathbb{P}^2$. More precisely:
				\begin{enumerate}
					\item{if $M$ is two-sided, then $M\simeq S^2,\R\mathbb{P}^2$;}
					\item{if $M$ is one-sided, then $M\simeq \R\mathbb{P}^2$.}	
				\end{enumerate}	Set $\rho:=\inf R_g$ then we have the following area bounds:
				\begin{itemize}
					\item{if $M\simeq S^2$, then $\h^2(M)\leq 8\pi/\rho;$}
					\item{if $M\simeq \R\mathbb{P}^2$, then $\h^2(M)\leq 4\pi/\rho$.}
					\end{itemize}
			\end{lem}	
			
			\begin{rmk}
				The conclusion of the lemma above is actually \textsl{false} if one only assumes $M$ to be stable (in lieu of the stronger condition that it arises as a geometric limit with higher multiplicity). For instance, in the product Riemannian manifold $\R\mathbb{P}^2\times S^1$ for any closed geodesic $\Gamma\subset \R\mathbb{P}^2$ the minimal surface $\Gamma\times S^1$ (a torus) is actually stable. 	
			\end{rmk}
			
			\begin{proof}
				
				If $M$ is two-sided then the conclusion comes straight from observing that such a $M$ must be stable (cf. \cite{Sha15, ACS15}). 
				
				 If instead $M$ is one-sided, we consider the construction of the twofold cover $\tilde{M}\subset\tilde{U}$.
				One still has that the convergence of $\tilde{M}_k$ to $\tilde{M}$ happens with the same multiplicity $m\geq 2$. Thus $\tilde{M}$ is stable and the argument above applies. Hence $\tilde{M}\simeq S^2$ or $\tilde{M}\simeq \R\mathbb{P}^2$ (the latter option to be ruled out by standard covering arguments) and thus $M\simeq \R\mathbb{P}^2$. 
				
				Let us now justify the area bounds. If $M$ is two-sided, then the statement comes directly from the stability inequality through the rearrangement trick by Schoen-Yau:
				
				\[
				\frac{1}{2}\int_{\Sigma}(R_g+|A^2|)\phi^2\,d \h^2 \leq\int_{\Sigma}|\nabla_{\Sigma}\phi|^2  \,d \h^2 +\int_{\Sigma}K\phi^2 \,d \h^2 
				\] 
				choosing the test function $\phi\equiv 1$ and recalling that the Gauss-Bonnet theorem gives
				\[
				\int_{\Sigma}K\,d \h^2 =\begin{cases}
				4\pi & \textrm{if} \ M\simeq S^2 \\
				2\pi & \textrm{if} \ M\simeq \R\mathbb{P}^2.
				\end{cases}
				\]
				In the case of a one-sided $\mathbb{R}\mathbb{P}^2$, we just need to notice it is covered twice by $\tilde{M}\subset \tilde{U}$ in the sense explained in the first part of the proof. The surface $\tilde{M}$ is two-sided and so the estimate above applies, which implies that its area is bounded by $8\pi/\rho$, as we had claimed. 
					\end{proof}

			\section{Strong compactness theorems}\label{sec:comp}
			
			We start with the proof of Theorem \ref{thm:conv1}.
			
			\begin{proof}
				First of all, let us recall that the work by Chodosh-Ketover-Maximo \cite{CKM15} ensures that, in a 3-manifold of positive scalar curvature, a uniform bound on the Morse index implies a bound on the area, so that given $j$ as in the statement we can find $\Lambda\in\mathbb{R}$ such that $\left\{M_k\right\}\subset \mathfrak{M}(\Lambda,j)$ and Theorem \ref{thm:main} is applicable.
				
				After passing to such a subsequence converging as in Theorem \ref{thm:main}, the argument is then divided in three cases, depending on what kind of bubbles arise. The first case yields the claimed result while the latter two cases will lead to contradictions.

				\

				\framebox[5.6cm]{Case $1$: There are no bubbles.}	\ In this case, $\mathcal{Y}=\emptyset$ and the convergence is smooth everywhere with finite multiplicity $m\geq 1$. Hence, equation \eqref{eq:geomquant} specifies in this case to $\chi(M_k)=m\chi(M)$, which holds true for all $k$ sufficiently large. Furthermore, if it were $m\geq 2$ then, by virtue of Lemma \ref{lem:stable-type}, we would infer that $\chi(M)\in\left\{1,2\right\}$ and in fact $\chi(M)=2$ in case $(N^3,g)$ does not contain any minimal $\R\mathbb{P}^2$. Therefore, the equation above gives a contradiction with the assumption that $\chi(M_k)<2$ (respectively an obvious contradiction in the latter case) and so we can exclude $m\geq 2$. In particular, in this case, we conclude the claimed smooth multiplicity one convergence.
				
				\
				
				\framebox[6.4cm]{Case $2$: All bubbles are catenoids.} \ In this case, we can assume that $m\geq 2$ (for else there would be no point of bad convergence, hence no bubbles, and the conclusion comes from Case $1$). Therefore we also have $\chi(M)\in\left\{1,2\right\}$ (and in fact $\chi(M)=2$ in case $(N^3,g)$ does not contain any minimal $\R\mathbb{P}^2$). Moreover, since all bubbles are catenoids, they all satisfy $\chi(\Sigma^y_{\ell})-b^y_{\ell}=-2$. Recalling Lemma \ref{lem:sumindex} and equation \eqref{eq:geomquant}, we obtain
     		\begin{equation}\label{eq:allcat}
     		\chi(M_k) =m\chi (M)+\sum_{y\in\mathcal{Y}}\sum_{\ell=1}^{J_y} (\chi(\Sigma^y_{\ell})-b^y_{\ell})
		= m\chi (M)-2\sum_{y\in\mathcal{Y}} J_y
		\geq m\chi (M)-2j.
     		\end{equation}	
				
This means that $\chi(M_k) \geq 2-2j$ and in fact $\chi(M_k) \geq 4-2j$ if $(N^3,g)$ does not contain any minimal $\R\mathbb{P}^2$, in contradiction to the assumption of the theorem.
				
				\
				
				\framebox[7.7cm]{Case $3$: There is a non-catenoidal bubble.} \ First, we notice that this case can only happen for $j=4$, which can be seen as follows. If $j\leq 3$ and there exist some bubbles, then by Lemma \ref{lem:sumindex} they must have index at most three. But by the results of Chodosh-Maximo \cite{CM14, CM18b}, there are no bubbles of index two or three, hence they must have index at most one. By the characterization of stable minimal surfaces in \cite{FCS80, dCP79, Pog81} bubbles cannot have index zero, hence all bubbles have index one and by \cite{LR89} are therefore catenoids, a contradiction.
								
				Furthermore, by virtue of \cite{Sc83}, no bubble has one end and bubbles with two ends must be catenoids. Thus, in this case the non-catenoidal bubble has at least three ends, which also implies that $m\geq 3$. We can now employ the estimate proven in \cite{CM18b} for embedded complete minimal surfaces $\Sigma^2\subset\R^3$ of finite Morse index
				\begin{equation}\label{eq:CMineq}
				3 index(\Sigma)\geq 2genus(\Sigma)+4b(\Sigma)-5,
				\end{equation}
				to conclude
\begin{equation}\label{eqCMapplied}
\begin{aligned}
				\sum_{y\in\mathcal{Y}}\sum_{\ell=1}^{J_y} (\chi(\Sigma^y_{\ell})-b^y_{\ell})	 &=\sum_{y\in\mathcal{Y}}\sum_{\ell=1}^{J_y} (2-2 genus(\Sigma^y_{\ell})-2b^y_{\ell})\\ 
				&\geq \sum_{y\in\mathcal{Y}}\sum_{\ell=1}^{J_y} (2b^y_{\ell}-3) - 3\sum_{y\in\mathcal{Y}}\sum_{\ell=1}^{J_y} index(\Sigma^y_{\ell}) \geq 3-3j.
\end{aligned}
\end{equation}
The last inequality follows from Lemma \ref{lem:sumindex} and the fact that $(2b^y_{\ell}-3)\geq 3$ for the non-catenoidal bubble (while $(2b^y_{\ell}-3)\geq 0$ for all other bubbles).

Hence, equation \eqref{eq:geomquant} yields $\chi(M_k) \geq 3\chi(M)+3-3j \geq 6-3j= 2-2j$ (since $j=4$) in the general case and $\chi(M_k) \geq 3\chi(M)+3-3j \geq 9-3j>4-2j$ (again using $j=4$) if $(N^3,g)$ does not contain any minimal $\R\mathbb{P}^2$, in contradiction to the assumption of the theorem.
\end{proof}
			
			When we deal with sequences of minimal surfaces having their Morse indices bounded by some integer $j$, possibly large, we can still prove a compactness theorem along the same lines:

			\begin{thm}\label{thm:conv2}
				Let $(N^3,g)$ be a compact Riemannian manifold, without boundary, of positive scalar curvature. Let $\left\{M_k\right\}$ be a sequence of closed, embedded minimal surfaces with Morse index bounded from above by $j\geq 5$. If
				\begin{enumerate}
					\item[\underline{\textsl{either}}]{$\chi(M_k)<6-3j$ for all $k\in\mathbb{N}$,}
					\item[\underline{\textsl{or}}]{$\chi(M_k)<9-3j $ for all $k\in\mathbb{N}$, and $N^3$ does not contain any embedded minimal $\R\mathbb{P}^2$ (which happens, for instance, if $N^3$ is simply connected),}
				\end{enumerate}
				then, up to extracting a subsequence, one has that $\left\{M_k\right\}$ converges smoothly to some closed, embedded minimal surface $M$, of Morse index bounded from above by $j$, with multiplicity one.
			\end{thm}

			\begin{proof}
				Once again, observe that the applicability of the bubbling analysis (Theorem \ref{thm:main}) is ensured by \cite{CKM15}. We can therefore proceed along the lines of the proof of Theorem \ref{thm:conv1}. Indeed, if there are no bubbles, then we obtain smooth convergence with multiplicity one, since $m\geq 2$ would force $\chi(M_k)\geq 2$ by \eqref{eq:geomquant}, while by assumption clearly $\chi(M_k)<0$. 
								
				In the remainder of the proof, we therefore assume, towards a contradiction, that $m\geq 2$ and the class of bubbles is not empty. 
				First, assume that all bubbles are catenoids. Following the argument of Case 2 from the previous proof, and using the assumption that $j\geq 5$, we conclude $\chi(M_k) \geq 2-2j > 6-3j$ as well as $\chi(M_k) \geq 4-2j \geq 9-3j$ if $(N^3,g)$ does not contain any minimal $\R\mathbb{P}^2$, thus the desired contradiction. 
				
				Hence it remains to consider the case where there is at least one non-catenoidal bubble. Exactly as in Case 3 of the previous proof, we conclude that such a bubble has at least three ends and we must have $m\geq 3$. Therefore, estimate \eqref{eqCMapplied} again gives us $\chi(M_k) \geq 3\chi(M)+3-3j \geq 6-3j$ and $\chi(M_k) \geq 3\chi(M)+3-3j \geq 9-3j$ if $(N^3,g)$ does not contain any minimal $\R\mathbb{P}^2$, yielding the final contradiction.
			\end{proof}		
			
		Let us now turn to the study of minimal surfaces having a uniform bound on their index and area, without any topological assumption. We prove Theorem \ref{thm:conv1'} from the introduction.
				
		\begin{proof}
		Here, we use the general index estimate \eqref{eq:CMineq} in the form
\begin{equation}
b(\Sigma)-1 \leq \frac{3}{4} index(\Sigma)+\frac{1}{4},
\end{equation}
in particular dropping the term involving the genus of $\Sigma$. Combining this estimate with the conclusion of Proposition \ref{pro:mult} we get that a converging sequence of minimal surfaces with index uniformly bounded by $j$ shall converge with multiplicity $m$ bounded from above via a linear function of $j$. We need to distinguish two cases depending on whether the limit $M$ is two-sided or one-sided. In the former case, using Lemma \ref{lem:sumindex}, we get
			\[
			m \leq  1 + \sum_{y\in\mathcal{Y}}\sum_{\ell=1}^{J_y} (b^y_\ell - 1)
			  \leq   1 + \sum_{y\in\mathcal{Y}}\sum_{\ell=1}^{J_y}\left(\frac{3}{4}index(\Sigma^y_l)+\frac{1}{4}\right)   
			 \leq   1+\frac{3}{4}j+\frac{1}{4}\sum_{y\in\mathcal{Y}}J_y\leq 1+j.
			\]
		
			Therefore, using Lemma \ref{lem:stable-type}, we can then derive (when $m\geq 2$) that
			\[
			\limsup_{k\to\infty}\h^2(M_k)\leq \frac{8\pi}{\rho}\left(1+j\right).
			\]
			
In the latter case, if $M$ is one-sided, one proves instead the inequality $m\leq 2+2j$. Again, by Lemma \ref{lem:stable-type}, (now with $M\simeq \R\mathbb{P}^2$), we can then derive (for $m\geq 2$) that
			\[
			\limsup_{k\to\infty}\h^2(M_k)\leq \frac{4\pi}{\rho}\left(2+2j\right) = \frac{8\pi}{\rho}\left(1+j\right).
			\]
In both cases we obtain a contradiction to the assumption and hence conclude that $m=1$ and the convergence is smooth.
			\end{proof}

			\section{Topological lower semicontinuity}
			
			This section is devoted to the presentation of the proof of Theorem \ref{thm:lsc}, which is in fact a fairly direct consequence of Proposition \ref{pro:mult}.
			
			\begin{proof}
			Let us recall from Corollary \ref{lem:geomquant} that for $k$ large we have
			\[
			\chi(M_k)  = m\chi(M) + \sum_{y\in\mathcal{Y}}\sum_{\ell=1}^{J_y} (2-2genus(\Sigma^y_\ell) - 2b^y_\ell).
			\] 
			Now we apply the multiplicity estimate, starting by dividing the above by two: 
			\begin{eqnarray*}
			1-genus(M_k) &=& m(1-genus(M)) - \sum_{y\in\mathcal{Y}}\sum_{\ell=1}^{J_y} genus(\Sigma^y_\ell) - \sum_{y\in\mathcal{Y}}\sum_{\ell=1}^{J_y} (b^y_\ell -1) \\
			&\leq & m(1-genus(M)) - \sum_{y\in\mathcal{Y}}\sum_{\ell=1}^{J_y} genus(\Sigma^{y}_{\ell})+1-m.
			\end{eqnarray*}
			The claimed inequality comes by simply rearranging the terms. 
			
			When equality occurs in \eqref{eq:lsc}, then in particular we must have equality in the multiplicity estimate, thus
			\[
			m=1+\sum_{y\in\mathcal{Y}}\sum_{\ell=1}^{J_y}(b^y_{\ell}-1).
			\]
				\end{proof}

			\begin{rmk}	\label{rem:nonorient}
				In case either the limit surface $M$ and/or the surfaces belonging to the sequence $\left\{M_k\right\}$ are non-orientable one can gain similar results by simply following the same argument modulo recalling the appropriate expression for the Euler characteristic, and using part (2) of Proposition \ref{pro:mult} in lieu of part (1).
			\end{rmk}

  \end{document}